\documentclass[11pt]{amsart}

\usepackage[margin=1in]{geometry}
\usepackage{amsmath,amssymb,amsthm,mathtools}
\usepackage{booktabs}
\usepackage{caption}
\usepackage{delimset}
\usepackage{enumitem}
\usepackage{float}
\usepackage{microtype}
\usepackage{tikz}
\usepackage[numbers,sort&compress,longnamesfirst]{natbib}
\usetikzlibrary{arrows.meta,calc,fit,positioning}
\usepackage{xcolor}
\definecolor{BIT}{cmyk}{1, 0, 1, 0}
\usepackage[colorlinks,linkcolor=BIT,citecolor=BIT,urlcolor=BIT]{hyperref}
\usepackage[capitalise,nameinlink]{cleveref}
\crefformat{section}{#2\S#1#3}
\Crefformat{section}{#2\S#1#3}
\crefformat{subsection}{#2\S#1#3}
\Crefformat{subsection}{#2\S#1#3}
\crefname{equation}{Eq.}{Eqs.}
\Crefname{equation}{Eq.}{Eqs.}
\crefname{figure}{Fig.}{Figs.}
\Crefname{figure}{Fig.}{Figs.}
\numberwithin{equation}{section}

\makeatletter
\let\originalNATtest\NAT@test
\renewcommand\NAT@test[1]{%
  \ifNAT@numbers
    \ifnum#1=\@ne
      \ifnum\NAT@ctype=\z@
        \NAT@hyper@{\originalNATtest{#1}}%
      \else
        \originalNATtest{#1}%
      \fi
    \else
      \originalNATtest{#1}%
    \fi
  \else
    \originalNATtest{#1}%
  \fi
}
\makeatother
\hypersetup{
  pdftitle={Infinite Families of Counterexamples to the Stanley--Gasharov Conjecture},
  pdfauthor={David G. L. Wang, K. Zhang, and T.Y. Zhao},
  pdfkeywords={Stanley--Gasharov conjecture, chromatic symmetric function, claw-free graph, Schur positivity, line graph}
}

\newtheorem{lemma}{Lemma}[section]
\newtheorem{theorem}[lemma]{Theorem}
\newtheorem{proposition}[lemma]{Proposition}
\newtheorem{conjecture}[lemma]{Conjecture}
\crefname{conjecture}{Conjecture}{Conjectures}
\Crefname{conjecture}{Conjecture}{Conjectures}
\theoremstyle{remark}
\newtheorem{remark}[lemma]{Remark}

\newcommand{\CSF}{X}
\newcommand{\Gtwo}{G_{2}}
\newcommand{\Gone}{G_{1}}
\newcommand{\Ft}{F_t}
\newcommand{\Qt}{Q_t}
\newcommand{\Bt}{B_t(G,\,v)}
\newcommand{\coeff}[2]{\brk[s]1{s_{#1}}\,#2}

\newcommand{\graphpointsize}{1.45pt}
\tikzset{
  graph edge/.style={line width=.35pt},
  graph label/.style={font=\scriptsize,inner sep=1pt},
  bridge/.style={line width=.55pt}
}
\newcommand{\GraphDot}[1]{\fill (#1) circle (\graphpointsize);}

\title[Two infinite counterexample families to the Stanley--Gasharov conjecture]
{Two infinite families of counterexamples\\
to the Stanley--Gasharov conjecture}
\author[D. G. L. Wang]{David G. L. Wang}
\address{School of Mathematics and Statistics \& MIIT Key Laboratory of Mathematical Theory and Computation in Information Security, Beijing Institute of Technology, Beijing 102400, P. R. China}
\email{glw@bit.edu.cn}

\author[K. Zhang]{K. Zhang}
\address{School of Mathematics and Statistics, Beijing Institute of Technology, Beijing 102400, P. R. China}
\email{kai@bit.edu.cn}

\author[T.Y. Zhao]{T.Y. Zhao$^*$}
\thanks{$^*$Corresponding author.}
\address{College of Science, China University of Petroleum (Beijing), Beijing 102249, P. R. China}
\email{zhaotongyuan@cup.edu.cn}
\thanks{Wang is supported by the National Natural Science Foundation of China
(Grant No.~12171034).}

\date{}

\begin{document}
\begin{abstract}
The Stanley--Gasharov conjecture asserts that every claw-free graph is Schur-positive. Prajapati and, independently, Matherne and Morales identified the same pair of counterexamples, both of which are line graphs, thereby disproving the conjecture. In this paper, we construct two infinite families of counterexamples to the Stanley--Gasharov conjecture, thereby answering a question of Matherne and Morales. Every graph in the first family is a line graph, whereas no graph in the second family is a line graph.

Prajapati further showed that the graph $G_2$, which has $12$ vertices and $21$ edges, is the smallest counterexample under the ordering that first compares the numbers of vertices and then the numbers of edges. We show that $G_2$ is also the smallest counterexample under the reverse ordering, which first compares the edge numbers and then the vertex numbers. Similarly, we exhibit a graph $Q$ with $13$ vertices and $27$ edges and show that $Q$ is the smallest counterexample that is not a line graph under each ordering. Our two infinite families are obtained from $G_2$ and $Q$, respectively, by adjoining a clique of order at least $4$ and connecting one of its vertices to a distinguished vertex of the original graph by a single edge.
\end{abstract}
\maketitle

\section{Introduction}\label{sec:introduction}

Throughout this paper, all graphs are finite and simple.
For a graph $G$, 
\citet{Stanley1995} introduced the \emph{chromatic symmetric
function} 
\[
\CSF_G\coloneqq
\sum_{\text{ proper }\kappa\colon V(G)\to\mathbb Z_{>0}}
\prod_{u\in V(G)}x_{\kappa(u)},
\]
as a symmetric-function refinement of Birkhoff's chromatic polynomial. 
It is now a central object in algebraic combinatorics. A symmetric function is \emph{$b$-positive} if all coefficients
in its $b$-expansion are nonnegative; here $b$ is any basis of the algebra of symmetric functions. A graph is \emph{$b$-positive} if
its chromatic symmetric function is $b$-positive.  The Schur-positivity is
representation-theoretically natural because Schur functions encode
irreducible polynomial representations; see
\citet[pp.~162--163]{Macdonald1995} and \citet{Stanley1981,Stanley1999}. \citet{Stanley1998} posed 
the following conjecture and attributed it
to Gasharov; see also \citet{Gasharov1999}.
\begin{conjecture}[\citeauthor{Stanley1998}--Gasharov, 1998]\label{conj:SG}
Every claw-free graph is Schur-positive.
\end{conjecture}
A principal source of evidence for \cref{conj:SG} was the theorem of
\citet{Gasharov1996} that the incomparability graph of every $(3+1)$-free
poset is Schur-positive; these graphs form a prominent claw-free subclass.
The stronger $e$-positivity assertion for this subclass, known as
the \emph{Stanley--Stembridge conjecture}
~\cite{StanleyStembridge1993} is recently proved by \citet{Hikita2025}.

Several algebraic approaches to the positivity problems of graphs have been developed.
\citet{GebhardSagan2001} introduced a noncommutative approach, while
\citet{ShareshianWachs2016} introduced chromatic quasisymmetric functions.
To attack \cref{conj:SG}, \citet{WangWang2020} gave a
combinatorial formula for every Schur coefficient of~$\CSF_G$ and used it
to classified the Schur-positive
complete bipartite graphs and complete tripartite graphs.  \citet{Pawlowski2022}
subsequently developed a group-algebra approach.
Beyond incomparability graphs, \citet{ShelburneVanWilligenburg2025} proved
the Schur positivity of a claw-free family called generalized nets. They \citep{ShelburneVanWilligenburg2026} subsequently characterized
Schur-positive complete multipartite graphs.

Days ago, \cref{conj:SG} was disproved independently by
\citet{Prajapati2026} and by \citet{MM2026}, both presenting the same pair of line graph counterexamples. Line graphs are all claw-free.
The complete census in
\citet[Proposition~4]{Prajapati2026} shows that the number~$12$ is the minimum
order of a counterexample graph and that exactly two counterexample graphs occur at that order.  
In the notation of \citet{Prajapati2026}, they are~$G$
and~$G'$,
with $22$ and $21$ edges,
respectively.
In the notation of
\citet{MM2026}, these graphs are
$\Gone=L(H_1)$ and $\Gtwo=L(H_2)$, 
respectively.
As a result, under the lexicographic order
\[
\brk1{\lvert V(G)\rvert,\,\lvert E(G)\rvert},
\]
the graph~$\Gtwo$ is the unique
\emph{vertex-first, edge-second minimum} counterexample.

\Citet[Remark~2.3]{MM2026} wrote
``It would be interesting to find ... infinite families of counterexamples.'' The census of \citet{Prajapati2026} determines the
vertex-first, edge-second minimum, whereas the complementary
edge-first, vertex-second minimum remained to be determined.  The
existence of infinite families also remains open.  

This paper is threefold.
\begin{enumerate}[label=\textup{(\arabic*)},leftmargin=*,itemsep=3pt]
\item 
In \cref{sec:edge-first-minimality}, we report an exhaustive exact
computation for the previously untreated range
\[
13\le n\le21,\qquad n-1\le m\le20, 
\]
where $n$ is the number of vertices and $m$ is the number of edges.
Our computation checks all $144{,}492$ connected claw-free graphs
and finds no counterexample.  Together with the census of
\citet{Prajapati2026}, this shows that the line graph~$\Gtwo$ is the
unique minimum under the complementary edge-first, vertex-second order.
\item In \cref{sec:line-family}, we extend the graph~$\Gtwo$ to graphs~$\Ft$, for
$t\ge4$, which form an infinite family of counterexamples. More precisely,
\[
\coeff{(4,\,4,\,4,\,4,\,1^{t-4})}{\CSF_{\Ft}}
=-40(t-1)(t-1)!<0.
\]

\item In \cref{sec:nonline-family}, we construct a counterexample graph~$Q$ that is not a line graph. It has $13$ vertices and~$27$ edges and satisfies
\[
\coeff{(3,\,3,\,3,\,3,\,1)}{\CSF_Q}=-144.
\]
Combined exhaustive exact
computations show that~$Q$ is the smallest graph under both the vertex-first, edge-second order and the edge-first, vertex-second order.  For~$t\ge5$, we construct the infinite family of counterexample graphs~$\Qt$ that are not 
line graphs, with
\[
\coeff{(4,\,4,\,4,\,4,\,2,\,1^{t-5})}{\CSF_{\Qt}}
=-144(t-1)(t-1)!<0.
\]
\end{enumerate}

To place the third contribution in the structural theory of claw-free graphs, recall that a graph is \emph{quasi-line} if the neighborhood of every vertex is the union of two cliques.  Line graphs form a proper subclass of quasi-line graphs, which in turn form a proper subclass of claw-free graphs. The structure theory of \citet{ChudnovskySeymour2005} describes connected claw-free graphs in terms of basic classes and expansion operations.

\section{The base counterexample and its edge-first minimality}
\label{sec:edge-first-minimality}
The $21$-edge counterexamples in
\citet[Proposition~4]{Prajapati2026} and \citet{MM2026} are isomorphic.
Following the notation of \citet{MM2026}, let~$H_2$ be the graph on
vertex set
$\{1,\,\dots,\,10\}$ with edge set
$E(H_2)=\{v_1,\,\dots,\,v_{12}\}$, where
\[
\begin{array}{r@{\,=\,}l r@{\,=\,}l r@{\,=\,}l r@{\,=\,}l
                r@{\,=\,}l r@{\,=\,}l}
v_1 & \{1,\,8\}, & v_2 & \{1,\,9\}, & v_3 & \{2,\,7\}, &
v_4 & \{3,\,6\}, & v_5 & \{3,\,10\}, & v_6 & \{4,\,5\},\\
v_7 & \{4,\,9\}, & v_8 & \{5,\,9\}, & v_9 & \{6,\,10\}, &
v_{10} & \{7,\,9\}, & v_{11} & \{7,\,10\}, & v_{12} & \{8,\,10\}.
\end{array}
\]
Let~$\Gtwo=L(H_2)$.  For $1\le i\le12$, the vertex~$v_i$ of~$\Gtwo$
corresponds to the edge~$v_i$ of~$H_2$ under the line-graph construction.
By the
forbidden-induced-subgraph characterization of \citet{Beineke1970}, every
line graph is claw-free; hence~$\Gtwo$ is claw-free.
See \cref{fig:G2}.
\begin{figure}[ht]
\centering
\begin{minipage}[b]{0.45\textwidth}
\centering
\begin{tikzpicture}[scale=0.82,line cap=round,line join=round]
  \coordinate (h1) at (-0.2,-1.4);
  \coordinate (h2) at ( 1.7, 1.6);
  \coordinate (h3) at (-1.5, 0.5);
  \coordinate (h4) at ( 0.8,-1.6);
  \coordinate (h5) at ( 1.8,-0.5);
  \coordinate (h6) at (-0.5, 1.6);
  \coordinate (h7) at ( 0.8, 0.5);
  \coordinate (h8) at (-1.1,-0.5);
  \coordinate (h9) at ( 0.8,-0.5);
  \coordinate (h10) at (-0.5, 0.5);

  \draw[graph edge] (h1)--(h8) (h1)--(h9)
        (h2)--(h7)
        (h3)--(h6) (h3)--(h10)
        (h4)--(h5) (h4)--(h9) (h5)--(h9)
        (h6)--(h10) (h7)--(h9) (h7)--(h10) (h8)--(h10);

  \foreach \p in {h1,h2,h3,h4,h5,h6,h7,h8,h9,h10} {\GraphDot{\p}}
  \node[graph label,below left=2.5pt] at (h1) {$1$};
  \node[graph label,above right=2.5pt] at (h2) {$2$};
  \node[graph label,left=2.5pt] at (h3) {$3$};
  \node[graph label,below=2.5pt] at (h4) {$4$};
  \node[graph label,right=2.5pt] at (h5) {$5$};
  \node[graph label,above=2.5pt] at (h6) {$6$};
  \node[graph label,above left=2.5pt] at (h7) {$7$};
  \node[graph label,above left=2.5pt] at (h8) {$8$};
  \node[graph label,above left=2.5pt] at (h9) {$9$};
  \node[graph label,above right=2.5pt] at (h10) {$10$};
\end{tikzpicture}
\captionsetup{font=footnotesize,justification=centering}
\captionof{figure}{The root graph~$H_2$.}
\label{fig:H2}
\end{minipage}\hfill
\begin{minipage}[b]{0.51\textwidth}
\centering
\begin{tikzpicture}[scale=0.90,line cap=round,line join=round]
  \coordinate (v1)  at (-0.65,-1.50);
  \coordinate (v2)  at (-1.35,-0.60);
  \coordinate (v3)  at (-0.65, 1.50);
  \coordinate (v4)  at ( 2.30, 0.00);
  \coordinate (v5)  at ( 1.30, 0.60);
  \coordinate (v6)  at (-3.50, 0.00);
  \coordinate (v7)  at (-2.50, 0.60);
  \coordinate (v8)  at (-2.50,-0.60);
  \coordinate (v9)  at ( 1.30,-0.60);
  \coordinate (v10) at (-1.35, 0.60);
  \coordinate (v11) at ( 0.15, 0.60);
  \coordinate (v12) at ( 0.15,-0.60);

  \draw[graph edge] (v1)--(v2) (v1)--(v12)
        (v2)--(v7) (v2)--(v8) (v2)--(v10)
        (v3)--(v10) (v3)--(v11) (v10)--(v11)
        (v4)--(v5) (v4)--(v9)
        (v5)--(v9) (v5)--(v11) (v5)--(v12)
        (v6)--(v7) (v6)--(v8)
        (v7)--(v8) (v7)--(v10)
        (v8)--(v10)
        (v9)--(v11) (v9)--(v12)
        (v11)--(v12);

  \foreach \p in {v1,v2,v3,v4,v5,v6,v7,v8,v9,v10,v11,v12} {\GraphDot{\p}}
  \node[graph label,below=2.5pt] at (v1) {$v_1$};
  \node[graph label,below left=2.5pt] at (v2) {$v_2$};
  \node[graph label,above=2.5pt] at (v3) {$v_3$};
  \node[graph label,right=2.5pt] at (v4) {$v_4$};
  \node[graph label,above=2.5pt] at (v5) {$v_5$};
  \node[graph label,left=2.5pt] at (v6) {$v_6$};
  \node[graph label,above=2.5pt] at (v7) {$v_7$};
  \node[graph label,below=2.5pt] at (v8) {$v_8$};
  \node[graph label,below=2.5pt] at (v9) {$v_9$};
  \node[graph label,above left=2.5pt] at (v10) {$v_{10}$};
  \node[graph label,above right=2.5pt] at (v11) {$v_{11}$};
  \node[graph label,below right=2.5pt] at (v12) {$v_{12}$};
\end{tikzpicture}
\captionsetup{font=footnotesize,justification=centering}
\captionof{figure}{The line graph~$\Gtwo=L(H_2)$.}
\label{fig:G2}
\end{minipage}
\end{figure}
The graph~$\Gtwo$ is isomorphic to the graph~$G'$ in
\citet[Proposition~4]{Prajapati2026} and to the $21$-edge
counterexample in \citet{MM2026}.  In particular,
\begin{equation}\label{eq:base-negative}
\coeff{(3,\,3,\,3,\,3)}{\CSF_{\Gtwo}}=-40.
\end{equation}
Thus the graph~$\Gtwo$ is a counterexample to \cref{conj:SG}.

\begin{proposition}[Edge-first minimality]\label{prop:edge-first-minimality}
Every claw-free graph~$H$ with fewer than $21$ edges, or with $21$ edges and
fewer than $12$ vertices, is Schur-positive.
The graph~$\Gtwo$ is the unique claw-free graph with
$21$ edges and $12$ vertices that is not Schur-positive.
Consequently, the graph $\Gtwo$ is the unique minimum counterexample to \cref{conj:SG} under the order
that first compares the number of edges and then the number of vertices.
\end{proposition}

We prove \cref{prop:edge-first-minimality} by the census described below.

\subsection{Finite search range}\label{subsec:finite-search}

It suffices to enumerate connected graphs.  Indeed, the chromatic symmetric
function is multiplicative over connected components, and a product of
Schur-positive symmetric functions is Schur-positive by the
\emph{Littlewood--Richardson rule}. Thus, if a graph is not Schur-positive, then one of its connected
components is not Schur-positive.

If a connected graph has $n$ vertices and $m$ edges, then~$n\le m+1$.
The complete census of \citet{Prajapati2026}, together with its
accompanying data, supplies every predecessor with~$n\le12$ and the
target stratum~$(m,n)=(21,12)$. A predecessor with~$n\ge13$ must
satisfy~$m\le20$, and connectedness then gives~$n\le m+1\le21$.
Consequently, the previously untreated connected range is exactly
\begin{equation}\label{eq:new-census-range}
13\le n\le21,\qquad n-1\le m\le20.
\end{equation}
For each order~$n$ in \eqref{eq:new-census-range}, we used version~2.9.0 of
the \texttt{geng} program of \citet{McKayPiperno2014}, with option
\texttt{-cF} and edge range~$n-1{:}20$, to generate one representative of
every connected claw-free isomorphism class.

\subsection{Exact Schur-coefficient calculation}
\label{subsec:exact-calculation}

For each graph~$G$ in \cref{eq:new-census-range}, we first computed its
power-sum expansion
using the spanning-subgraph formula of \citet{Stanley1995}:
\begin{equation}\label{eq:power-sum-census}
\CSF_G=\sum_{A\subseteq E(G)}(-1)^{\lvert A\rvert}p_{\lambda_A},
\end{equation}
where $\lambda_A$ is the partition of $\lvert V(G)\rvert$ formed by the
component sizes of the spanning subgraph $(V(G),\,A)$.  We then converted
to the Schur basis using the \emph{Frobenius character formula}; see
\citet[p.~114]{Macdonald1995}:
\begin{equation}\label{eq:frobenius-census}
p_\mu=\sum_{\lambda\vdash \lvert V(G)\rvert}
\chi^\lambda(\mu)s_\lambda.
\end{equation}
Here~$\chi^\lambda(\mu)$ is the value on cycle type~$\mu$ of the
irreducible character of~$S_{\lvert V(G)\rvert}$ indexed by~$\lambda$.
Thus, if $b_\mu$ denotes the coefficient of $p_\mu$ in
\cref{eq:power-sum-census}, then every Schur coefficient is computed as
\[
\coeff{\lambda}{\CSF_G}
=\sum_{\mu\vdash \lvert V(G)\rvert}b_\mu\chi^\lambda(\mu).
\]

The coefficients~$b_\mu$ were computed by a frontier-connectivity dynamic
program.  Fix an ordering of the edges.  Between two consecutive edge steps,
the \emph{frontier} consists of the vertices incident with both a processed
and an unprocessed edge.  After each prefix, a state groups the choices~$A$
of included edges in that prefix that induce the same data.  For any such
choice, each block of the frontier partition is the intersection of the
frontier with one component of the graph on the introduced vertices with
edge set~$A$.  The integer attached to the block is the total size of that
partial component, including its forgotten vertices.  The sizes of
components with no active vertex form the partition of \emph{closed
components}.  The state weight is the sum of~$(-1)^{\lvert A\rvert}$ over
all choices~$A$ represented by the state.

We maintain this invariant using the following operational convention.  A
vertex is introduced as a singleton partial component immediately before
its first incident edge is processed.  It remains \emph{active} while it has
an unprocessed incident edge and is forgotten immediately after its final
incident edge is processed.  The endpoints of the edge currently being
processed are therefore temporarily active even if they do not belong to
the frontier immediately before or after that step.  In particular, a
degree-one vertex is introduced immediately before its unique edge and
forgotten immediately afterward.  When a vertex is forgotten, it is removed
from its active block without changing the stored component size.  If it is
the final active vertex of its partial component, that size is appended to
the partition of closed components.  A closed component cannot later
reconnect: once its final active vertex has been forgotten, no unprocessed
edge is incident with any vertex of that component.
\begingroup
\small
\begin{verbatim}
for n = 13 to 21:
    generate each connected claw-free G with n - 1 <= m <= 20
    for each G:
        D = {empty active state: 1}
        for each edge uv, in order:
            introduce with stored size 1 each endpoint whose first edge is uv
            for each (state, weight) in D:
                add weight to the state excluding uv
                add -weight to the state including uv
                    and merging the blocks of u and v,
                    adding their sizes when the blocks are distinct
            forget each endpoint whose final edge is uv
                if its partial component has no active vertex:
                    append its stored size to the closed partition
            canonicalize states and combine equal states
        b[mu] = sum of the terminal weights with component partition mu
        c[lambda] = sum_mu b[mu] chi^lambda(mu)
                    for every partition lambda of n
        if c[lambda] < 0 for some lambda:
            record G and all negative coefficients
\end{verbatim}
\endgroup

Every spanning subgraph follows a unique sequence of exclusion and inclusion
transitions, and its terminal weight is~$(-1)^{\lvert A\rvert}$.  Combining
equal canonical states is valid because such states have identical possible
continuations.  At termination every vertex has been forgotten, so the
closed-component partition is precisely~$\lambda_A$.  The terminal weights
therefore give the coefficients in \cref{eq:power-sum-census} exactly.  We
generated the character values
$\chi^\lambda(\mu)$ by the \emph{Murnaghan--Nakayama rule}; see
\citet[p.~117]{Macdonald1995}, and then applied
\cref{eq:frobenius-census}.  All computations used exact integer arithmetic.

\subsection{Census results}\label{subsec:census-results}

The combined totals are shown in \cref{tab:edge-first-census}.  Our
computation found no negative Schur coefficient in the previously untreated
range~\cref{eq:new-census-range}.

\begin{table}[ht]
\centering
\small
\begin{tabular}{rcrrc}
\toprule
$n$ & permitted $m$ & connected claw-free graphs & counterexamples & source\\
\midrule
1  & $0$       & $1$      & $0$ & Prajapati\\
2  & $1$       & $1$      & $0$ & Prajapati\\
3  & $2$--$3$    & $2$      & $0$ & Prajapati\\
4  & $3$--$6$    & $5$      & $0$ & Prajapati\\
5  & $4$--$10$   & $14$     & $0$ & Prajapati\\
6  & $5$--$15$   & $50$     & $0$ & Prajapati\\
7  & $6$--$21$   & $191$    & $0$ & Prajapati\\
8  & $7$--$21$   & $813$    & $0$ & Prajapati\\
9  & $8$--$21$   & $2{,}271$  & $0$ & Prajapati\\
10 & $9$--$21$  & $5{,}965$  & $0$ & Prajapati\\
11 & $10$--$21$  & $14{,}314$ & $0$ & Prajapati\\
12 & $11$--$20$  & $20{,}765$ & $0$ & Prajapati\\
13 & $12$--$20$  & $31{,}693$ & $0$ & this work\\
14 & $13$--$20$  & $38{,}229$ & $0$ & this work\\
15 & $14$--$20$  & $35{,}278$ & $0$ & this work\\
16 & $15$--$20$  & $23{,}972$ & $0$ & this work\\
17 & $16$--$20$  & $11{,}339$ & $0$ & this work\\
18 & $17$--$20$  & $3{,}400$  & $0$ & this work\\
19 & $18$--$20$  & $545$    & $0$ & this work\\
20 & $19$--$20$  & $35$     & $0$ & this work\\
21 & $20$      & $1$      & $0$ & this work\\
\midrule
\multicolumn{2}{r}{Combined predecessor total}
& $188{,}884$ & $0$ & combined\\
\midrule
12 & $21$ (target stratum) & $8{,}907$ & $1$ & Prajapati\\
\bottomrule
\end{tabular}
\caption{Numbers and sources of connected claw-free isomorphism classes in
the combined edge-first census.}
\label{tab:edge-first-census}
\end{table}

The rows with~$n\le12$ and the target stratum were extracted from the
census files accompanying \citet{Prajapati2026}; the article itself
reports the complete totals by order and the two counterexamples of
order~$12$. The nine rows with~$13\le n\le21$ contain exactly the
$144{,}492$ graphs checked in this work.  None has a negative Schur
coefficient.  In the target stratum~$n=12$ and~$m=21$, the accompanying
census data identify~$\Gtwo$ as the unique counterexample. 

\begin{proof}[Proof of \Cref{prop:edge-first-minimality}]
The complete census of \citet{Prajapati2026}, together with its
accompanying data, covers all connected predecessors with~$n\le12$ and
identifies~$\Gtwo$ as the unique counterexample in the target stratum.  Our
exact computation covers every remaining connected predecessor by
\cref{eq:new-census-range} and finds no counterexample among its $144{,}492$
classes.  These two computations cover every connected predecessor.
Disconnected graphs are covered by the component reduction in
\cref{subsec:finite-search}.
\end{proof}

\section{An infinite family of line-graph counterexamples}
\label{sec:line-family}

In this section and the next, we construct two infinite families of
counterexamples by adjoining a complete graph through a single bridge.
We begin with the line-graph family and prove a Schur-coefficient transfer
lemma that will also be used in \cref{sec:nonline-family}.  The
construction is motivated by conjoined graphs studied by \citet{QTW26}.

For every integer~$t\ge 4$, take a copy of~$K_t$ disjoint from~$\Gtwo$ and
choose a distinguished vertex~$w\in V(K_t)$.  Define the graph~$\Ft$ by
adding the single edge~$v_3w$:
\begin{equation}\label{eq:family-definition}
\Ft\coloneqq\brk1{\Gtwo\sqcup K_t}+v_3w.
\end{equation}
The edge~$v_3w$ is the \emph{bridge} joining~$\Gtwo$ and~$K_t$.  In particular,
\[
\lvert V(\Ft)\rvert=12+t
\quad\text{and}\quad
\lvert E(\Ft)\rvert=22+\binom{t}{2}.
\]
See \cref{fig:general-construction}.
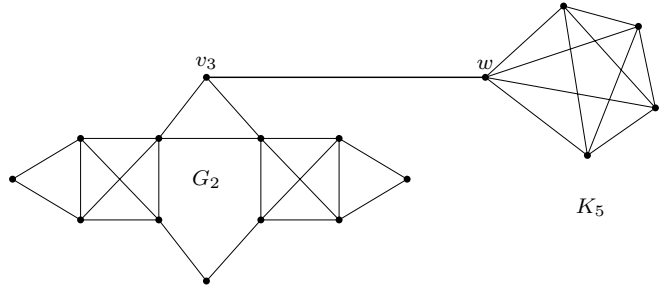
\begin{figure}[H]
\centering
\begin{tikzpicture}[scale=0.90,line cap=round,line join=round]
  \coordinate (gv1)  at (-2.30,-1.50);
  \coordinate (gv2)  at (-3.00,-0.60);
  \coordinate (gv3)  at (-2.30, 1.50);
  \coordinate (gv4)  at ( 0.65, 0.00);
  \coordinate (gv5)  at (-0.35, 0.60);
  \coordinate (gv6)  at (-5.15, 0.00);
  \coordinate (gv7)  at (-4.15, 0.60);
  \coordinate (gv8)  at (-4.15,-0.60);
  \coordinate (gv9)  at (-0.35,-0.60);
  \coordinate (gv10) at (-3.00, 0.60);
  \coordinate (gv11) at (-1.50, 0.60);
  \coordinate (gv12) at (-1.50,-0.60);

  \draw[graph edge] (gv1)--(gv2) (gv1)--(gv12)
        (gv2)--(gv7) (gv2)--(gv8) (gv2)--(gv10)
        (gv3)--(gv10) (gv3)--(gv11) (gv10)--(gv11)
        (gv4)--(gv5) (gv4)--(gv9)
        (gv5)--(gv9) (gv5)--(gv11) (gv5)--(gv12)
        (gv6)--(gv7) (gv6)--(gv8)
        (gv7)--(gv8) (gv7)--(gv10)
        (gv8)--(gv10)
        (gv9)--(gv11) (gv9)--(gv12)
        (gv11)--(gv12);

  \coordinate (w) at (1.80,1.50);
  \coordinate (c2) at (2.95,2.55);
  \coordinate (c3) at (4.05,2.25);
  \coordinate (c4) at (4.30,1.05);
  \coordinate (c5) at (3.30,0.35);
  \draw[graph edge] (w)--(c2) (w)--(c3) (w)--(c4) (w)--(c5)
        (c2)--(c3) (c2)--(c4) (c2)--(c5)
        (c3)--(c4) (c3)--(c5) (c4)--(c5);

  \draw[bridge] (gv3)--(w);
  \foreach \p in {gv1,gv2,gv3,gv4,gv5,gv6,gv7,gv8,gv9,gv10,gv11,gv12,w,c2,c3,c4,c5} {\GraphDot{\p}}
  \node[graph label,above=2pt] at (gv3) {$v_3$};
  \node[graph label,above=2pt] at (w) {$w$};
  \node[graph label] at (-2.30,0.00) {$\Gtwo$};
  \node[graph label] at (3.35,-0.42) {$K_5$};
\end{tikzpicture}
\caption{The construction of~$\Ft$, illustrated for~$t=5$: the full copy
of~$\Gtwo$ is joined to a copy of~$K_5$ by the single edge~$v_3w$.}
\label{fig:general-construction}
\end{figure}

\begin{proposition}\label{prop:line-graph}
For every integer~$t\ge 4$, the graph~$\Ft$ is a connected line graph.
Consequently, $\Ft$ is claw-free.
\end{proposition}

\begin{proof}
Let~$\Ft'$ be the graph with
\[
V(\Ft')=V(H_2)\sqcup\{c,\,u_2,\,\dots,\,u_t\}
\quad\text{and}\quad
E(\Ft')=E(H_2)\cup\{\{2,\,c\}\}
\cup\bigl\{\{c,\,u_i\}\colon 2\le i\le t\bigr\}.
\]
It is routine to check that~$\Ft\cong L(\Ft')$.  Since~$H_2$ is connected,
so is~$\Ft'$.  Thus~$\Ft$ is a connected line graph and hence is claw-free.
\end{proof}

To transfer the negative Schur coefficient from~$\Gtwo$ to~$\Ft$, we prove
a general statement that does not require the graph~$G$ to be claw-free.
Let~$G$ be any graph, let~$v\in V(G)$, take a disjoint copy of~$K_t$ with a
distinguished vertex~$w$, and define
\[
\Bt\coloneqq\brk1{G\sqcup K_t}+vw.
\]
Let~$\Lambda$ be the ring of symmetric functions over~$\mathbb Z$, and let
\[
\rho_t\colon\Lambda\longrightarrow
\mathbb Z[x_1,\,\dots,\,x_t]^{\mathfrak S_t}
\]
be the specialization obtained by setting $x_{t+1}=x_{t+2}=\dots=0$.
Equivalently, $\rho_t\brk1{\CSF_G}$ is the generating function for proper
colorings of~$G$ using only the colors~$1,\,\dots,\,t$.

\begin{lemma}\label{lem:specialization}
For every graph~$G$, every vertex~$v\in V(G)$, and every positive
integer~$t$,
\begin{equation}\label{eq:specialization-identity}
\rho_t\brk1{\CSF_{\Bt}}
=(t-1)(t-1)!\,e_t(x_1,\,\dots,\,x_t)\,
\rho_t\brk1{\CSF_G}.
\end{equation}
\end{lemma}

\begin{proof}
Fix a proper coloring~$\kappa\colon V(G)\to\{1,\,\dots,\,t\}$.  In every
extension~$\widehat\kappa$ of~$\kappa$ to~$\Bt$, the vertices of~$K_t$ must
receive each of the~$t$ colors exactly once.
Since~$vw$ is an edge, the color of~$w$ must differ from~$\kappa(v)$.
There are therefore exactly~$t-1$ choices for~$\widehat\kappa(w)$; this is
the factor~$t-1$ in \cref{eq:specialization-identity}.  After this color is
chosen, the remaining~$t-1$ colors can be assigned bijectively to the
remaining~$t-1$ vertices of~$K_t$ in~$(t-1)!$ ways.

Every such extension contributes the monomial
\[
\prod_{u\in V(G)}x_{\kappa(u)}\prod_{i=1}^t x_i
=e_t(x_1,\,\dots,\,x_t)\prod_{u\in V(G)}x_{\kappa(u)}.
\]
Consequently, each proper $t$-coloring of~$G$ contributes
$(t-1)(t-1)!$ extensions, all with the common additional factor
$e_t(x_1,\,\dots,\,x_t)$.  Summing over all proper $t$-colorings
of~$G$ gives \cref{eq:specialization-identity}.
\end{proof}

We next define the partition indexing the transferred coefficient.  For a
partition~$\lambda$, let~$\ell\brk1{\lambda}$ denote its number of parts.
Write $\lambda=(\lambda_1,\,\dots,\,\lambda_\ell)$, where
$\ell=\ell\brk1{\lambda}\le t$.  Pad
$\lambda$ with zeros until it has exactly~$t$ entries, and define
\begin{equation}\label{eq:lambda-plus-column}
\lambda+(1^t)
\coloneqq(\lambda_1+1,\,\dots,\,\lambda_t+1).
\end{equation}
In Young-diagram language, this operation adds one full column of height~$t$.

\begin{lemma}[Schur-coefficient transfer]\label{lem:transfer}
Let~$\lambda\vdash \lvert V(G)\rvert$ and suppose that
$\ell\brk1{\lambda}\le t$.  Then
\[
\coeff{\lambda+(1^t)}{\CSF_{\Bt}}
=(t-1)(t-1)!\,
\coeff{\lambda}{\CSF_G}.
\]
\end{lemma}

\begin{proof}
By \cref{lem:specialization}, we work in~$t$ variables.  Write
\[
\CSF_G=\sum_{\mu\vdash \lvert V(G)\rvert}a_\mu s_\mu.
\]
Only the terms with~$\ell\brk1{\mu}\le t$ remain after specialization.  Since
$e_t(x_1,\,\dots,\,x_t)=x_1\dotsm x_t$, the tableau definition of Schur
polynomials gives
\[
\begin{aligned}
e_t(x_1,\,\dots,\,x_t)s_\mu(x_1,\,\dots,\,x_t)
&=s_{\mu+(1^t)}(x_1,\,\dots,\,x_t)
\qquad\text{if }\ell\brk1{\mu}\le t.
\end{aligned}
\]
The map $\mu\mapsto\mu+(1^t)$ is injective, so the coefficient in question
is~$a_\lambda$.  \cref{lem:specialization} now gives
$(t-1)(t-1)!a_\lambda$.  Because the Schur polynomials
$s_\nu(x_1,\,\dots,\,x_t)$ with $\ell\brk1{\nu}\le t$ are linearly independent,
this is also the coefficient of~$s_{\lambda+(1^t)}$ in~$\CSF_{\Bt}$.
\end{proof}

We now apply \cref{lem:transfer} to the negative coefficient in
\cref{eq:base-negative}.  Take
\[
\lambda=(3,\,3,\,3,\,3).
\]
Since $\ell\brk1{\lambda}=4$, the lemma applies for every $t\ge 4$.  Padding with
zeros and adding $(1^t)$ gives
\[
\lambda+(1^t)=(4,\,4,\,4,\,4,\,1^{t-4}).
\]
Consequently,
\begin{equation}\label{eq:family-negative}
\coeff{(4,\,4,\,4,\,4,\,1^{t-4})}{\CSF_{\Ft}}
=(t-1)(t-1)!\,
\coeff{(3,\,3,\,3,\,3)}{\CSF_{\Gtwo}}
=-40(t-1)(t-1)!<0.
\end{equation}

\begin{theorem}\label{thm:infinite-family}
For every integer~$t\ge4$, the graph~$\Ft$ is a connected line graph and
\[
\coeff{(4,\,4,\,4,\,4,\,1^{t-4})}{\CSF_{\Ft}}
=-40(t-1)(t-1)!<0.
\]
Consequently, the graphs~$\Ft$, for~$t\ge4$, form an infinite family of counterexamples to \cref{conj:SG}.
\end{theorem}

\begin{proof}
\Cref{prop:line-graph} shows that each~$\Ft$ is a connected line graph and
hence is claw-free.  The displayed coefficient is negative
by \eqref{eq:family-negative}.  The graphs are pairwise nonisomorphic because
$\lvert V(\Ft)\rvert=12+t$ is strictly increasing with~$t$.
\end{proof}

The lower bound~$t\ge4$ in \Cref{thm:infinite-family} is sharp.  Indeed, the
same construction defines~$F_3$, and the root-graph construction in the
proof of \cref{prop:line-graph} shows that~$F_3$ is a line graph.  A direct
Maple computation using the symmetric-function package of
\citet{Stembridge1995} shows that~$F_3$ is Schur-positive.

\section{An infinite family of non-line-graph counterexamples}
\label{sec:nonline-family}

The family constructed in \cref{sec:line-family} consists of line graphs.
We now construct a claw-free counterexample~$Q$ that is not a line graph,
establish its uniqueness and minimality under both lexicographic orders,
and apply \cref{lem:transfer} to obtain an infinite family.

\subsection{Minimality under both lexicographic orders}
\label{subsec:nonline-minimality}

The $22$-edge counterexamples in
\citet[Theorem~1]{Prajapati2026} and \citet{MM2026} are isomorphic.
Following the notation of \citet{MM2026}, let~$H_1$ be the graph on
vertex set
$\{1,\,\dots,\,10\}$ with edge set
$E(H_1)=\{u_1,\,\dots,\,u_{12}\}$, where
\[
\begin{array}{r@{\,=\,}l r@{\,=\,}l r@{\,=\,}l r@{\,=\,}l
                r@{\,=\,}l r@{\,=\,}l}
u_1 & \{1,\,8\}, & u_2 & \{7,\,10\}, & u_3 & \{2,\,7\}, &
u_4 & \{4,\,5\}, & u_5 & \{4,\,9\}, & u_6 & \{3,\,6\},\\
u_7 & \{6,\,10\}, & u_8 & \{3,\,10\}, & u_9 & \{5,\,9\}, &
u_{10} & \{8,\,10\}, & u_{11} & \{8,\,9\}, &
u_{12} & \{7,\,9\}.
\end{array}
\]
Let~$\Gone=L(H_1)$, where the vertex~$u_i$ of~$\Gone$ corresponds to the
edge~$u_i$ of~$H_1$ under the line-graph construction; see
\cref{fig:G1}.

\begin{figure}[ht]
\centering
\begin{minipage}[b]{0.45\textwidth}
\centering
\begin{tikzpicture}[scale=0.82,line cap=round,line join=round]
  \coordinate (h1) at (-1.5,-1.6);
  \coordinate (h2) at ( 1.7, 1.6);
  \coordinate (h3) at (-1.5, 0.5);
  \coordinate (h4) at ( 0.8,-1.6);
  \coordinate (h5) at ( 1.8,-0.5);
  \coordinate (h6) at (-0.5, 1.6);
  \coordinate (h7) at ( 0.8, 0.5);
  \coordinate (h8) at (-0.5,-0.5);
  \coordinate (h9) at ( 0.8,-0.5);
  \coordinate (h10) at (-0.5, 0.5);

  \draw[graph edge] (h1)--(h8)
        (h2)--(h7)
        (h3)--(h6) (h3)--(h10)
        (h4)--(h5) (h4)--(h9) (h5)--(h9)
        (h6)--(h10) (h7)--(h9) (h7)--(h10) 
        (h8)--(h10) (h8)--(h9);

  \foreach \p in {h1,h2,h3,h4,h5,h6,h7,h8,h9,h10} {\GraphDot{\p}}
  \node[graph label,below left=2.5pt] at (h1) {$1$};
  \node[graph label,above right=2.5pt] at (h2) {$2$};
  \node[graph label,left=2.5pt] at (h3) {$3$};
  \node[graph label,below=2.5pt] at (h4) {$4$};
  \node[graph label,right=2.5pt] at (h5) {$5$};
  \node[graph label,above=2.5pt] at (h6) {$6$};
  \node[graph label,above left=2.5pt] at (h7) {$7$};
  \node[graph label,above left=2.5pt] at (h8) {$8$};
  \node[graph label,above left=2.5pt] at (h9) {$9$};
  \node[graph label,above right=2.5pt] at (h10) {$10$};
\end{tikzpicture}
\captionsetup{font=footnotesize,justification=centering}
\captionof{figure}{The root graph~$H_1$.}
\label{fig:H1}
\end{minipage}\hfill
\begin{minipage}[b]{0.51\textwidth}
\centering
\begin{tikzpicture}[scale=0.90,line cap=round,line join=round]
  \coordinate (u3)  at (-0.65,-1.50);
  \coordinate (u2)  at (-1.35,-0.60);
  \coordinate (u1)  at (-0.65, 1.50);
  \coordinate (u4)  at ( 2.30, 0.00);
  \coordinate (u5)  at ( 1.30, 0.60);
  \coordinate (u6)  at (-3.50, 0.00);
  \coordinate (u7)  at (-2.50, 0.60);
  \coordinate (u8)  at (-2.50,-0.60);
  \coordinate (u9)  at ( 1.30,-0.60);
  \coordinate (u10) at (-1.35, 0.60);
  \coordinate (u11) at ( 0.15, 0.60);
  \coordinate (u12) at ( 0.15,-0.60);

  \draw[graph edge] (u3)--(u2) (u3)--(u12)
        (u2)--(u7) (u2)--(u8) (u2)--(u10) (u2)--(u12) 
        (u1)--(u10) (u1)--(u11) (u10)--(u11)
        (u4)--(u5) (u4)--(u9)
        (u5)--(u9) (u5)--(u11) (u5)--(u12)
        (u6)--(u7) (u6)--(u8)
        (u7)--(u8) (u7)--(u10)
        (u8)--(u10)
        (u9)--(u11) (u9)--(u12)
        (u11)--(u12);

  \foreach \p in {u1,u2,u3,u4,u5,u6,u7,u8,u9,u10,u11,u12} {\GraphDot{\p}}
  \node[graph label,below=2.5pt] at (u3) {$u_3$};
  \node[graph label,below left=2.5pt] at (u2) {$u_2$};
  \node[graph label,above=2.5pt] at (u1) {$u_1$};
  \node[graph label,right=2.5pt] at (u4) {$u_4$};
  \node[graph label,above=2.5pt] at (u5) {$u_5$};
  \node[graph label,left=2.5pt] at (u6) {$u_6$};
  \node[graph label,above=2.5pt] at (u7) {$u_7$};
  \node[graph label,below=2.5pt] at (u8) {$u_8$};
  \node[graph label,below=2.5pt] at (u9) {$u_9$};
  \node[graph label,above left=2.5pt] at (u10) {$u_{10}$};
  \node[graph label,above right=2.5pt] at (u11) {$u_{11}$};
  \node[graph label,below right=2.5pt] at (u12) {$u_{12}$};
\end{tikzpicture}
\captionsetup{font=footnotesize,justification=centering}
\captionof{figure}{The line graph~$\Gone=L(H_1)$.}
\label{fig:G1}
\end{minipage}
\end{figure}

Define~$Q$ by adjoining a new vertex~$u_{13}$ to~$\Gone$
and setting
\begin{equation}\label{eq:nonline-base-definition}
\begin{aligned}
V(Q)&=V(\Gone)\sqcup\{u_{13}\},\\
E(Q)&=E(\Gone)\cup
\{u_{13}u_4,\,u_{13}u_5,\,u_{13}u_9,\,
  u_{13}u_{11},\,u_{13}u_{12}\}.
\end{aligned}
\end{equation}
Thus~$u_5$ and~$u_{13}$ are true twins; explicitly,
\begin{equation}\label{eq:true-twins}
N_Q[u_5]=N_Q[u_{13}]
=\{u_4,\,u_5,\,u_9,\,u_{11},\,u_{12},\,u_{13}\}.
\end{equation}
The graph~$Q$ has $13$ vertices and $27$ edges; see
\cref{fig:nonline-base}.

\begin{figure}[ht]
\centering
\begin{tikzpicture}[scale=0.90,line cap=round,line join=round]
  \coordinate (u3)  at (-0.65,-1.50);
  \coordinate (u2)  at (-1.35,-0.60);
  \coordinate (u1)  at (-0.65, 1.50);
  \coordinate (u4)  at ( 3.30, 0.00);
  \coordinate (u5)  at ( 1.30, 0.60);
  \coordinate (u6)  at (-3.50, 0.00);
  \coordinate (u7)  at (-2.50, 0.60);
  \coordinate (u8)  at (-2.50,-0.60);
  \coordinate (u9)  at ( 1.30,-0.60);
  \coordinate (u10) at (-1.35, 0.60);
  \coordinate (u11) at ( 0.15, 0.60);
  \coordinate (u12) at ( 0.15,-0.60);
  \coordinate (u13) at ( 1.80, 0.00);

  \draw[graph edge] (u3)--(u2) (u3)--(u12)
        (u2)--(u7) (u2)--(u8) (u2)--(u10) (u2)--(u12) 
        (u1)--(u10) (u1)--(u11) (u10)--(u11)
        (u4)--(u5) (u4)--(u9)
        (u5)--(u9) (u5)--(u11) (u5)--(u12)
        (u6)--(u7) (u6)--(u8)
        (u7)--(u8) (u7)--(u10)
        (u8)--(u10)
        (u9)--(u11) (u9)--(u12)
        (u11)--(u12)
        (u13)--(u11) (u13)--(u12) (u13)--(u4) (u13)--(u5)
        (u13)--(u9);

  \foreach \p in {u1,u2,u3,u4,u5,u6,u7,u8,u9,u10,u11,u12,u13} {\GraphDot{\p}}
  \node[graph label,above=2.5pt] at (u1) {$u_1$};
  \node[graph label,below left=2.5pt] at (u2) {$u_2$};
  \node[graph label,below=2.5pt] at (u3) {$u_3$};
  \node[graph label,right=4.5pt] at (u4) {$u_4$};
  \node[graph label,above=2.5pt] at (u5) {$u_5$};
  \node[graph label,left=2.5pt] at (u6) {$u_6$};
  \node[graph label,above=2.5pt] at (u7) {$u_7$};
  \node[graph label,below=2.5pt] at (u8) {$u_8$};
  \node[graph label,below=2.5pt] at (u9) {$u_9$};
  \node[graph label,above left=2.5pt] at (u10) {$u_{10}$};
  \node[graph label,above right=2.5pt] at (u11) {$u_{11}$};
  \node[graph label,below right=2.5pt] at (u12) {$u_{12}$};
  \node[graph label,below right=2.5pt] at (u13) {$u_{13}$};
\end{tikzpicture}
\captionsetup{font=footnotesize,justification=centering}
\caption{The graph~$Q$, obtained from~$\Gone$ by adjoining the true twin
$u_{13}$ of~$u_5$.}
\label{fig:nonline-base}
\end{figure}
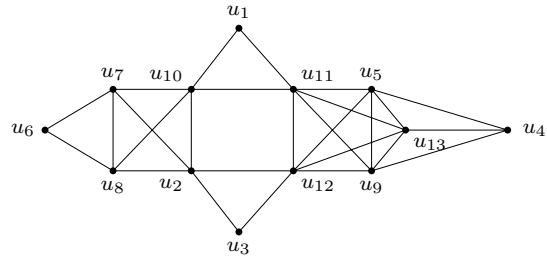

\begin{proposition}\label{prop:nonline-base}
The graph~$Q$ is connected and claw-free, but is not a line graph.  Moreover,
\begin{equation}\label{eq:nonline-base-negative}
\coeff{(3,\,3,\,3,\,3,\,1)}{\CSF_Q}=-144.
\end{equation}
\end{proposition}

\begin{proof}
Since~$\Gone$ is connected and claw-free, adjoining the true twin~$u_{13}$
of~$u_5$ preserves these two properties; hence~$Q$ is connected and
claw-free.  The set
\[
W=\{u_4,\,u_5,\,u_9,\,u_{11},\,u_{13}\}
\]
induces~$K_5-u_4u_{11}$, which is not the line graph of a simple graph:
if its two nonadjacent vertices corresponded to disjoint edges~$ab$ and~$cd$,
the other three would have to correspond to three pairwise intersecting
edges chosen from~$ac,ad,bc,bd$, which is impossible.  Since line
graphs are closed under taking induced subgraphs, $Q$ is not a line graph.
A direct exact computation using the method of
\cref{subsec:exact-calculation} gives
\cref{eq:nonline-base-negative}.
\end{proof}

We next establish the minimality of~$Q$ under both lexicographic orders. We first record the vertex-first census.  By
\citet[Proposition~4]{Prajapati2026}, every connected claw-free graph on
at most~$11$ vertices is Schur-positive, and the only connected
claw-free graphs on~$12$ vertices whose chromatic symmetric functions
are not Schur-positive are the two classes represented by~$\Gtwo$
and~$\Gone$.  Both are line graphs.

Starting from the complete order-$12$ census reported in
\citet{Prajapati2026} and its accompanying data, we generated every
connected claw-free graph on~$13$ vertices with at most~$27$ edges by
adjoining one vertex with every possible nonempty neighborhood and
removing duplicate isomorphism classes by exact canonical labeling.
This is exhaustive because every connected graph has a non-cut vertex
and claw-freeness is hereditary.  The exact Schur computations are
summarized in \cref{tab:nonline-minimality-census}.

\begin{table}[H]
\centering
\small
\begingroup
\setlength{\tabcolsep}{5pt}
\begin{tabular}{@{}crrr@{}}
\toprule
range
& \shortstack[c]{connected claw-free\\graphs}
& counterexamples
& \shortstack[c]{counterexamples that are\\not line graphs}\\
\midrule
$n=12$, $11\le m\le66$
& $1{,}728{,}404$
& $2$
& $0$\\
$n=13$, $12\le m\le26$
& $251{,}199$
& $0$
& $0$\\
$n=13$, $m=27$
& $58{,}577$
& $1$
& $1$\\
\bottomrule
\end{tabular}
\endgroup
\caption{The exact census used for vertex-first, edge-second minimality
of~$Q$.  The entries count isomorphism classes.}
\label{tab:nonline-minimality-census}
\end{table}

For the edge-first order, we use two elementary bounds.  If a connected
claw-free graph has~$n$ vertices and~$m$ edges and is not a line graph,
then
\[
m\ge n+1
\qquad\text{and}\qquad
m\ge7.
\]
Indeed, the cases $m=n-1$ and $m=n$ give, respectively, a claw-free tree
and a claw-free unicyclic graph; these are line graphs.  The second bound
follows from Beineke's characterization, since each non-claw obstruction
has at least seven edges.

The census in \cref{sec:edge-first-minimality} covers every predecessor
with at most~$20$ edges, while
\cref{tab:nonline-minimality-census} covers all relevant graphs with at
most~$13$ vertices and the target stratum.  Hence the remaining connected
strict predecessors of~$Q$ lie exactly in
\begin{equation}\label{eq:nonline-edge-first-range}
14\le n\le25,\qquad
\max\{21,n+1\}\le m\le26.
\end{equation}
For this range, we used \texttt{geng} from nauty~2.9.0 with option
\texttt{-cFq}, removed the line graphs by exact recognition, and computed
the complete Schur expansion of every remaining graph by the method in
\cref{subsec:exact-calculation}.  The order-$13$ strata with
$21\le m\le27$ were also rerun.

\begin{table}[H]
\centering
\small
\begingroup
\setlength{\tabcolsep}{5pt}
\begin{tabular}{@{}crrr@{}}
\toprule
range
& \shortstack[c]{connected claw-free\\classes generated}
& \shortstack[c]{classes that are not\\line graphs}
& counterexamples\\
\midrule
$n=13$, $21\le m\le27$
& $278{,}083$
& $214{,}376$
& $1$\\
\eqref{eq:nonline-edge-first-range}
& $15{,}809{,}909$
& $6{,}505{,}691$
& $0$\\
\midrule
Total
& $16{,}087{,}992$
& $6{,}720{,}067$
& $1$\\
\bottomrule
\end{tabular}
\endgroup
\caption{The exact census used for edge-first, vertex-second minimality
among claw-free graphs that are not line graphs.  The entries count
isomorphism classes.}
\label{tab:nonline-edge-first-census}
\end{table}

The unique counterexample in
\cref{tab:nonline-edge-first-census} occurs at $(n,m)=(13,27)$, and exact
canonical labeling identifies it with~$Q$.  As independent checks, all
$53$ generator scopes were recounted without sharding, and the line-graph
filter and Schur-coefficient program were compared with independent
implementations on fixed verification sets.

\begin{proposition}[Minimum non-line-graph counterexample]
\label{prop:nonline-minimality}
Let~$G$ be a claw-free graph that is not a line graph and is not
Schur-positive.  Then
\[
\lvert V(G)\rvert\ge13
\qquad\text{and}\qquad
\lvert E(G)\rvert\ge27.
\]
The graph~$Q$ is the unique counterexample of this type with $13$
vertices and $27$ edges.  Consequently, $Q$ is the unique minimum under
both lexicographic orders
\[
\bigl(\lvert V(G)\rvert,\,\lvert E(G)\rvert\bigr)
\quad\text{and}\quad
\bigl(\lvert E(G)\rvert,\,\lvert V(G)\rvert\bigr).
\]
\end{proposition}

\begin{proof}
For connected graphs, \citet[Proposition~4]{Prajapati2026} and
\cref{tab:nonline-minimality-census} give the vertex bound and identify
$Q$ as the unique counterexample in the target stratum.  The census in
\cref{sec:edge-first-minimality}, the two bounds above, and
\cref{tab:nonline-edge-first-census} exclude every connected graph of the
required type with fewer than~$27$ edges.

Now let~$G$ be disconnected, and choose a connected component~$C$ whose
chromatic symmetric function is not Schur-positive.  If
$\lvert V(G)\rvert\le13$, then~$C$ has at most~$12$ vertices and hence is
$\Gtwo$ or~$\Gone$.  It uses~$12$ vertices, so every remaining component
is isolated; thus~$G$ is a line graph, a contradiction.

For the edge bound, if~$C$ is not a line graph, the connected case gives
$\lvert E(C)\rvert\ge27$.  If~$C$ is a line graph, some other component
$D$ is not a line graph.  By \cref{prop:edge-first-minimality},
$\lvert E(C)\rvert\ge21$, while the bound above gives
$\lvert E(D)\rvert\ge7$.  Hence
\[
\lvert E(G)\rvert\ge21+7=28.
\]
The asserted bounds and uniqueness now follow.
\end{proof}

\begin{remark}\label{rem:Q-structure}
The graph~$Q$ is quasi-line.  Indeed, it is obtained from the line
graph~$\Gone$ by replacing the vertex~$u_5$ with the clique
$\{u_5,u_{13}\}$ and giving the two vertices the same neighbors outside
the clique.  Moreover,
\[
\bigl(\{u_5,u_{13}\},\,\{u_4\}\bigr)
\]
is a homogeneous pair of cliques.  Thus~$Q$ is an expansion of a line-graph core in the claw-free structure theory of
\citet{ChudnovskySeymour2005}.
\end{remark}

\subsection{The infinite family}
\label{subsec:nonline-infinite-family}

For every integer~$t\ge5$, take a copy of~$K_t$ disjoint from~$Q$ and choose
a distinguished vertex~$w\in V(K_t)$.  Define
\begin{equation}\label{eq:nonline-family-definition}
\Qt\coloneqq\brk1{Q\sqcup K_t}+u_1w.
\end{equation}
Thus
\[
\lvert V(\Qt)\rvert=13+t
\quad\text{and}\quad
\lvert E(\Qt)\rvert=28+\binom{t}{2}.
\]
See \cref{fig:nonline-construction}.

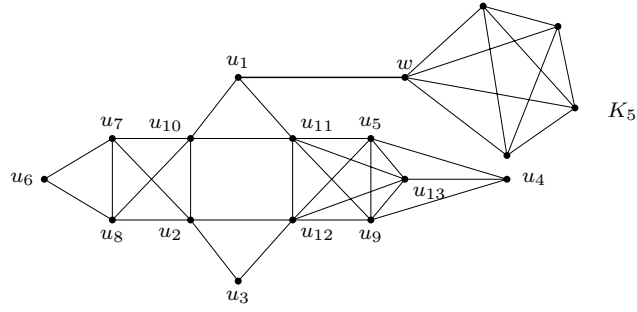
\begin{figure}[ht]
\centering
\begin{tikzpicture}[scale=0.90,line cap=round,line join=round]
 \coordinate (u3)  at (-0.65,-1.50);
  \coordinate (u2)  at (-1.35,-0.60);
  \coordinate (u1)  at (-0.65, 1.50);
  \coordinate (u4)  at ( 3.30, 0.00);
  \coordinate (u5)  at ( 1.30, 0.60);
  \coordinate (u6)  at (-3.50, 0.00);
  \coordinate (u7)  at (-2.50, 0.60);
  \coordinate (u8)  at (-2.50,-0.60);
  \coordinate (u9)  at ( 1.30,-0.60);
  \coordinate (u10) at (-1.35, 0.60);
  \coordinate (u11) at ( 0.15, 0.60);
  \coordinate (u12) at ( 0.15,-0.60);
  \coordinate (u13) at ( 1.80, 0.00);

  \draw[graph edge] (u3)--(u2) (u3)--(u12)
        (u2)--(u7) (u2)--(u8) (u2)--(u10) (u2)--(u12) 
        (u1)--(u10) (u1)--(u11) (u10)--(u11)
        (u4)--(u5) (u4)--(u9)
        (u5)--(u9) (u5)--(u11) (u5)--(u12)
        (u6)--(u7) (u6)--(u8)
        (u7)--(u8) (u7)--(u10)
        (u8)--(u10)
        (u9)--(u11) (u9)--(u12)
        (u11)--(u12)
        (u13)--(u11) (u13)--(u12) (u13)--(u4) (u13)--(u5)
        (u13)--(u9);

  \foreach \p in {u1,u2,u3,u4,u5,u6,u7,u8,u9,u10,u11,u12,u13} {\GraphDot{\p}}
  \node[graph label,below=2.5pt] at (u3) {$u_3$};
  \node[graph label,below left=2.5pt] at (u2) {$u_2$};
  \node[graph label,above=2.5pt] at (u1) {$u_1$};
  \node[graph label,right=4.5pt] at (u4) {$u_4$};
  \node[graph label,above=2.5pt] at (u5) {$u_5$};
  \node[graph label,left=2.5pt] at (u6) {$u_6$};
  \node[graph label,above=2.5pt] at (u7) {$u_7$};
  \node[graph label,below=2.5pt] at (u8) {$u_8$};
  \node[graph label,below=2.5pt] at (u9) {$u_9$};
  \node[graph label,above left=2.5pt] at (u10) {$u_{10}$};
  \node[graph label,above right=2.5pt] at (u11) {$u_{11}$};
  \node[graph label,below right=2.5pt] at (u12) {$u_{12}$};
  \node[graph label,below right=2.5pt] at (u13) {$u_{13}$};
  
\coordinate (w) at (1.80,1.50);
  \coordinate (c2) at (2.95,2.55);
  \coordinate (c3) at (4.05,2.25);
  \coordinate (c4) at (4.30,1.05);
  \coordinate (c5) at (3.30,0.35);
  \draw[graph edge] (w)--(c2) (w)--(c3) (w)--(c4) (w)--(c5)
        (c2)--(c3) (c2)--(c4) (c2)--(c5)
        (c3)--(c4) (c3)--(c5) (c4)--(c5);

\draw[bridge] (u1)--(w);
  \foreach \p in {w,c2,c3,c4,c5} {\GraphDot{\p}}
  \node[graph label,above=2pt] at (w) {$w$};
  \node[graph label] at (5,1) {$K_5$};
\end{tikzpicture}
\captionsetup{font=footnotesize,justification=centering}
\caption{The construction of~$\Qt$, illustrated for~$t=5$: the graph~$Q$
is joined to a copy of~$K_5$ by the single edge~$u_1w$.}
\label{fig:nonline-construction}
\end{figure}

We apply \cref{lem:transfer} with~$G=Q$, $v=u_1$, and
\[
\lambda=(3,\,3,\,3,\,3,\,1).
\]
Since~$\ell\brk1{\lambda}=5$, the lemma applies for every~$t\ge5$.  Padding
with zeros and adding~$(1^t)$ gives
\[
\lambda+(1^t)=(4,\,4,\,4,\,4,\,2,\,1^{t-5}).
\]
Consequently,
\begin{equation}\label{eq:nonline-family-negative}
\coeff{(4,\,4,\,4,\,4,\,2,\,1^{t-5})}{\CSF_{\Qt}}
=(t-1)(t-1)!\,
\coeff{(3,\,3,\,3,\,3,\,1)}{\CSF_Q}
=-144(t-1)(t-1)!<0.
\end{equation}

\begin{theorem}\label{thm:nonline-family}
For every integer~$t\ge5$, the graph~$\Qt$ is connected and claw-free, but
is not a line graph, and
\[
\coeff{(4,\,4,\,4,\,4,\,2,\,1^{t-5})}{\CSF_{\Qt}}
=-144(t-1)(t-1)!<0.
\]
Consequently, the graphs~$\Qt$, for~$t\ge5$, form an infinite family of
claw-free non-line-graph counterexamples to \cref{conj:SG}.
\end{theorem}

\begin{proof}
The bridge~$u_1w$ makes~$\Qt$ connected.  The only possible new claw centers
are~$u_1$ and~$w$, but
\[
N_{\Qt}(u_1)=\{u_{10},\,u_{11},\,w\},
\qquad
N_{\Qt}(w)=\{u_1\}\cup\bigl(V(K_t)\setminus\{w\}\bigr),
\]
where~$u_{10}u_{11}\in E(Q)$ and~$V(K_t)\setminus\{w\}$ is a clique; hence
$\Qt$ is claw-free.  Since~$Q$ is an induced subgraph of~$\Qt$,
\cref{prop:nonline-base} implies that~$\Qt$ is not a line graph.  The
coefficient is given by \eqref{eq:nonline-family-negative}, and the graphs are
pairwise nonisomorphic because~$\lvert V(\Qt)\rvert=13+t$.
\end{proof}

\end{document}